\documentclass[a4paper,10pt]{amsart}

\usepackage[utf8]{inputenc}
\usepackage[english]{babel}
\usepackage[T1]{fontenc}
\usepackage{amsfonts}
\usepackage{amsmath}
\usepackage{amsthm}
\usepackage{amssymb}
\usepackage{hyperref}
\usepackage{syntonly}
\usepackage{mathtools}
\usepackage{algorithm}
\usepackage{algpseudocode}
\usepackage{tikz}
\usepackage{dsfont}
\usepackage{braket}
\usepackage{faktor}
\usepackage{array}
\usepackage{lscape}
\usepackage{rotating}
\usepackage{epigraph}
\usepackage{subfig}
\usepackage{indentfirst}
\usepackage{booktabs}
\usepackage{multirow}
\usepackage{caption}
\usepackage{tabu}
\usepackage{tabularx}
\usepackage{xcolor}
\usepackage{pgf}
\usepackage{tikz}
\usepackage{caption}
\usepackage{color}
\usepackage{enumitem}
\usepackage{mathtools}
\usepackage{float}
\usepackage{graphicx}
\usepackage[displaymath, tightpage]{preview}
\usepackage[toc,page]{appendix}
\usetikzlibrary{matrix}
\usetikzlibrary{arrows}
\hypersetup{
    colorlinks,
    linkcolor={red!50!black},
    citecolor={blue!50!black},
    urlcolor={blue!80!black}
}

\usepackage[colorinlistoftodos, textwidth=4cm, shadow,disable]{todonotes}

\theoremstyle{plain}
\newtheorem{theorem}{Theorem}[section]
\newtheorem{lemma}{Lemma}[section]
\newtheorem{proposition}{Proposition}[section]
\newtheorem{corollary}{Corollary}[section]

\theoremstyle{definition}
\newtheorem{definition}{Definition}[section]

\theoremstyle{remark}
\newtheorem{remark}{Remark}[section]

\title[Diffeomorphism type of small hyperplane arrangements]{The diffeomorphism 
type of small hyperplane arrangements is combinatorially determined}

\author[Matteo Gallet]{Matteo Gallet}
\address{(Matteo Gallet) RICAM, Austrian Academy of Sciences, Altenberger Stra{\ss}e 69, A-4040, 
Linz, AT.}
\email{matteo.gallet@ricam.oeaw.ac.at}
\author[Elia Saini]{Elia Saini}
\address{(Elia Saini) Department of Mathematics, University of Fribourg, Chemin 
du Mus\'ee 23, CH-1700, Fribourg, CH.}
\email{elia.saini@unifr.ch}

\begin{document}

\begin{abstract}

It is known that there exist hyperplane arrangements with same underlying matroid that admit 
non-homotopy equivalent complement manifolds. In this work we show that, in any rank, complex 
central hyperplane arrangements with up to 7 hyperplanes and same underlying 
matroid are isotopic. In particular, the diffeomorphism type of the complement 
manifold and the Milnor fiber and fibration of these arrangements are 
combinatorially determined, that is, they depend uniquely on the underlying 
matroid. To do this, we associate to every such matroid a topological space, 
that we call the \emph{reduced realization space}; its connectedness --- showed 
by means of symbolic computation --- implies the desired result.

\end{abstract}

\maketitle

\section*{Introduction}
\label{intro}
The central problem in hyperplane arrangement theory is to determine whether the 
topology or the homotopy type of the complement manifold of an arrangement is 
described by the combinatorial properties of the arrangement itself. 
This theory was first developed in~\cite{Arn69} with motivations in the 
study of configuration spaces.

One of the seminal works on the homotopy theory of complex hyperplane 
arrangements is the computation of the integer cohomology algebra structure of 
the complement manifold of an arrangement by Orlik and Solomon~\cite{OS80}. 
Motivated by work of Arnol'd, they exploited techniques of 
Brieskorn~\cite{Bri73} to provide a presentation of this cohomology algebra, in 
terms of generators and relations, that depends uniquely on the underlying 
matroid of the arrangement. 

The result of~\cite{OS80} has generated a lot of new conjectures and problems, 
asking which homotopy invariants of the complement manifold of an arrangement 
are combinatorially determined. A cornerstone in this direction is the isotopy 
theorem proved by Randell in~\cite{Ran89}. It states that the diffeomorphism 
type of the complement manifold does not change through an isotopy, that is a 
smooth one-parameter family of arrangements with constant underlying matroid. 
Afterwards, in~\cite{Ran97} Randell showed similar results for more 
sophisticated invariants such as the Milnor fiber and fibration of an 
arrangement (compare Definition~\ref{Milnor}).

Randell's isotopy theorem can be actually reformulated in terms of matroid realization 
spaces, that are related to the well studied matroid stratification of the 
Grassmannian. In their celebrated paper~\cite{GGMS87}, Gel'fand, Goresky, 
MacPherson and Serganova studied this stratification and described some of its 
equivalent reformulations. In particular, Randell's results give rise to the 
problem of describing the connected components of the matroid strata of the 
Grassmannian.

On the other hand, in~\cite{Ryb11} Rybnikov found an example of arrangements 
with same underlying matroid but non-isomorphic fundamental groups of the 
corresponding complement manifolds. However, in many remarkable cases the 
topology of the complement manifold can be still recovered simply by the 
combinatorial data. Thus, one important problem
is to characterize wider families of arrangements for which Randell's
isotopy theorem holds.

Several results in this direction appeared in the literature. 
In particular, Jiang and Yau~\cite{Jy98}, Nazir and Yoshinaga~\cite{NY12} and 
Amram, Teicher and Ye~\cite{ATY13} focused their attention on some specific 
classes of line arrangements in the complex projective plane. 
One-parameter families of isotopic arrangements have also been studied 
in~\cite{WY07}, \cite{WY08} and~\cite{YY09}. 
However, the techniques developed in these works 
seem hardly generalizable to higher dimensions.

To every matroid~$M$ we can associate the set of hyperplane arrangements 
having~$M$ as underlying matroid. Such set has a natural topological structure 
as a subset of a space of matrices, and it is called the \emph{realization 
space} of~$M$. Here, building on previous results of Delucchi and the 
first-named author (see~\cite{DS15}), we associate to a matroid another 
topological space, called its \emph{reduced realization space} 
(Definition~\ref{reducedrealizationspace}). As the name suggests, the latter is 
a subset of the realization space, and it is obtained by considering hyperplane 
arrangements of a given shape. Such shape is determined by what we call the 
\emph{normal frame} of a matrix (Definition~\ref{normalframe}). Exploiting some 
ideas from~\cite{BL73} and~\cite{Rui13} we study this space, and finally we 
describe (Proposition~\ref{keytool}) how the connectedness of the reduced 
realization space is related to the one of the ``classic'' realization space. 
Moreover, we show by means of symbolic computation and elementary algebraic 
geometry arguments that for any matroid with ground set of up to~$7$ elements 
the associated reduced realization space is either empty or connected.

So that, from the results of~\cite{Ran89} and~\cite{Ran97} we can 
conclude that the diffeomorphism type of the complement manifold 
and the Milnor fiber and fibrations of complex central hyperplane arrangements 
with up to 7 hyperplane are combinatorially determined, that is, they depend uniquely 
on the underlying matroid of these arrangements.

\smallskip

\noindent \textbf{Overview.} Section~\ref{Section1} contains some basic 
definitions on matroids and complex hyperplane arrangements. In 
Section~\ref{Section2}, we introduce the normal frame of a matrix and the 
reduced realization space of a matroid, and we deduce some of their properties. 
Section~\ref{Section3} is devoted to applications in the study of the isotopy 
type of arrangements with up to 7 hyperplanes. For readability’s sake we 
postpone some of the technical computations to Appendix~\ref{AppendixA}.
  
\smallskip 

\noindent \textbf{Acknowledgments}
We thank Emanuele Delucchi, Peter Michor, 
Fernando Muro, Rita Pardini and Peter Scheiblechner for the opportunity to 
discuss part of our work with them. We want to mention Beno\^{i}t 
Guerville-Ball\'{e} and Torsten Hoge for the very helpful conversations we had 
at the conference ``Hyperplane Arrangements and Reflection Groups'' held in 
Hannover in August 2015. Moreover, we thank Michael Falk and Masahiko Yoshinaga 
for feedback about the results of Section~\ref{Section3}. Finally, we would like 
to thank the anonymous referee for its patient work and its very useful suggestions.
The first-named author is supported by the Austrian Science Fund (FWF): W1214-N15/DK9 and P26607 - 
``Algebraic Methods in Kinematics: Motion Factorisation and Bond Theory''.
The second-named author is supported by the Swiss National Science Foundation 
grant PP00P2\_150552/1. 
 
\section{Matroids and arrangements}\label{Section1}

In this section we provide a quick review of some basic definitions and results 
about matroids and arrangements. We refer to the book~\cite{Oxl92} for a 
detailed treatment of matroid theory and we point to~\cite{OT92} for a general 
theory of arrangements and to~\cite{FR00} for a survey of their homotopy theory.
   
 \subsection{Matroids}

A \emph{matroid}~$M$ is a pair $(E,\mathfrak{I})$, 
where $E$ is a finite \emph{ground set} and $\mathfrak{I}\subseteq 2^{E}$ is a 
family of subsets of $E$ satisfying the following three conditions:
\begin{enumerate}[label=(I\arabic{*})]
   \item \label{I1} $\emptyset\in\mathfrak{I}$;
   \item \label{I2} If $I\in\mathfrak{I}$ and $J\subseteq I$, then 
                    $J\in\mathfrak{I}$;
   \item \label{I3} If $I$ and $J$ are in $\mathfrak{I}$ and $|I|<|J|$, then there 
                    exists an element $e\in J\setminus I$ 
                    such that $I\cup\{e\}\in\mathfrak{I}$.
  \end{enumerate}
The elements of~$\mathfrak{I}$ are called the \emph{independent sets} of~$M$.
Maximal independent sets (with respect to inclusion) are called \emph{bases}, 
and the set of bases of $M$ will be denoted by~$\mathfrak{B}$. 
By definition, the \emph{rank} of a subset $S\subseteq E$ is
\begin{equation*}
 \operatorname{rk}(S)=\max \bigl\{
                         |S\cap B| \,
                         \mid \,
                         B\in\mathfrak{B}
                        \bigr\}
\end{equation*}
and the \emph{rank} of the matroid~$M$ is the rank of the ground 
set~$E$. 

A rank~$d$ matroid $M$ with ground set $E=\{1,\ldots,m\}$ is 
called \emph{realizable} over~$\mathbb{C}$ if 
there exists a matrix $A\in M_{d,m}(\mathbb{C})$ of $d$ rows and $m$ columns 
with complex coefficients such that 
\begin{equation*}
 \left\lbrace
   J\subseteq E \,
   \left| \,
     \left\lbrace
       A^{j}
     \right\rbrace_{j\in J}
   \right.
   \text{is linearly independent over }\mathbb{C}
 \right\rbrace
\end{equation*}
is the family of independent sets of~$M$. Here, $A^{j}$ denotes the 
$j\text{-th}$ column of~$A$. We say that $A$ \emph{realizes}~$M$ 
over~$\mathbb{C}$. 

\begin{definition}[Realization space]\label{realizationspace}
 For a rank~$d$ matroid~$M$ with ground set $E=\{1,\ldots,m\}$ 
 the \emph{realization space} of~$M$ over~$\mathbb{C}$ is  
 the set $\mathcal{R}_{\mathbb{C}}(M)$ of matrices $A\in M_{d,m}(\mathbb{C})$ 
 that satisfy the following condition: 
 \begin{itemize}
  \item $A$ realizes $M$ over $\mathbb{C}$.
 \end{itemize}
We endow $\mathcal{R}_{\mathbb{C}}(M)$ with the subspace topology 
of~$M_{d,m}(\mathbb{C})$.
\end{definition}

If $\mathcal{R}_{\mathbb{C}}(M)$ is empty, that is, there are no matrices that 
realize~$M$ over~$\mathbb{C}$, we say that $M$ is \emph{non-realizable} 
over~$\mathbb{C}$.
   
 \subsection{Arrangements}
   
Any finite collection $\mathcal{A}=\{H_{1},\ldots,H_{m}\}$ of affine subspaces 
in $\mathbb{C}^{d}$ will be called an \emph{arrangement}.
Its \emph{complement manifold} $M(\mathcal{A})$ is the complement 
of the union of the~$H_{i}$ in~$\mathbb{C}^{d}$. 
The arrangement is \emph{central} if every $H_{i}$ contains the origin.  
For an arrangement $\mathcal{A}=\{H_{1},\ldots, H_{m}\}$ in~$\mathbb{C}^{d}$ 
we assign a \emph{rank} to each subset 
$S\subseteq\{1,\ldots,m\}$ by setting
\begin{equation*}
\operatorname{rk}_{\mathcal{A}}(S) = \operatorname{codim}\bigcap_{i\in S}H_{i}
\end{equation*}
(where we define the empty set to have codimension $d+1$).

We say that the arrangements $\mathcal{A}=\{H_{1},\ldots, H_{m}\}$ 
and $\mathcal{B}=\{K_{1},\ldots,K_{m}\}$
have the same \emph{combinatorial type} if the functions 
$\operatorname{rk}_{\mathcal A}$ and $\operatorname{rk}_{\mathcal B}$ coincide. 

Given an open interval $(a,b)\subseteq\mathbb{R}$, a \emph{smooth one-parameter} 
family of arrangements is a collection $\{\mathcal{A}_{t}\}_{t\in(a,b)}$ of 
arrangements $\mathcal{A}_{t}=\{H_{1}(t),\ldots,H_{m}(t)\}$ in $\mathbb{C}^{d}$ 
such that there exist smooth functions from~$(a,b)$ to~$\mathbb{C}$ for the 
coefficients of the defining equations of the subspaces $H_{i}(t)$. With a 
slight abuse of notation we write~$\mathcal{A}_{t}$ for 
$\{\mathcal{A}_{t}\}_{t\in(a,b)}$, omitting the interval of parameters~$(a,b)$.

\begin{definition}[Isotopic arrangements]\label{isotopic}
 A smooth one-parameter family of arrangements~$\mathcal{A}_{t}$ 
 is an \emph{isotopy} if for any~$t_{1}$ and~$t_{2}$ the 
 arrangements~$\mathcal{A}_{t_{1}}$ and~$\mathcal{A}_{t_{2}}$ have the same 
 combinatorial type.
 In this case we say that $\mathcal{A}_{t_{1}}$ and $\mathcal{A}_{t_{2}}$ are 
\emph{isotopic}.
\end{definition}

The following theorem, sometimes referred to as ``isotopy theorem'' was proved 
by Randell~\cite{Ran89}. This is one of the pillars on which our work is based, 
allowing us, from now on, 
to focus on isotopic arrangements.

\begin{theorem}[{\cite{Ran89}}]\label{RIT}
 If $\mathcal{A}_{t_1}$ and $\mathcal{A}_{t_2}$ are isotopic arrangements, then 
 the complement manifolds $M(\mathcal{A}_{t_{1}})$ and $M(\mathcal{A}_{t_{2}})$ 
are diffeomorphic.
\end{theorem}

A \emph{hyperplane arrangement} is an arrangement of codimension $1$ subspaces. 
Again, a hyperplane arrangement is \emph{central} if each of its subspaces is 
linear. For a central hyperplane arrangement 
$\mathcal{A}=\{H_{1},\ldots,H_{m}\}$ in $\mathbb{C}^{d}$ pick linear 
forms~$\alpha_{i}$ in the dual space~$(\mathbb{C}^{d})^{\ast}$ with 
$H_{i}=\ker\alpha_{i}$. The \emph{underlying matroid} of~$\mathcal{A}$ is by 
definition the matroid~$M_{\mathcal{A}}$ with ground set 
$E_{\mathcal{A}}=\{1,\ldots,m\}$ and 
\begin{equation*}
\mathfrak{I}_{\mathcal{A}}=\left\lbrace
                             S\subseteq E \,
                              \left| \,
                               \left\lbrace
                                  \alpha_{i}
                               \right\rbrace_{i\in S} \,
                                 \text{is linearly independent over }\mathbb{C}
                              \right.
                          \right\rbrace
\end{equation*}
as independent sets. Clearly, the matroid $M_{\mathcal{A}}$ does not depend on 
the choice of the linear forms~$\alpha_{i}$. The \emph{rank} of~$\mathcal{A}$ is 
by definition the rank of~$M_{\mathcal{A}}$ and we say that $\mathcal{A}$ is 
\emph{essential} if its rank is maximal.

Notice that a smooth one-parameter family~$\mathcal{A}_{t}$ of central 
hyperplane arrangements is an isotopy if and only if $M_{\mathcal{A}_{t_{1}}} = 
M_{\mathcal{A}_{t_{2}}}$ for any $t_{1}$ and $t_{2}$. 

\begin{definition}[Milnor fiber and fibration]\label{Milnor}
 Given linear forms $\alpha_{i}\in(\mathbb{C}^{d})^{\ast}$ with 
$H_{i}=\ker\alpha_{i}$, the polynomial $Q_{\mathcal{A}} = 
\prod_{i=1}^{m}\alpha_{i}$ is homogeneous of degree~$m$ and can be considered as 
a map 
\begin{equation*}
 Q_{\mathcal{A}}:M(\mathcal{A})\longrightarrow\mathbb{C}^{\ast}
\end{equation*}
that is the projection of a fiber bundle called the \emph{Milnor fibration} of 
the arrangement (see~\cite{Milnor1968}). The \emph{Milnor fiber} is then the 
fiber $F_{\mathcal{A}}=Q_{\mathcal{A}}^{-1}(1)$. 
\end{definition}

The subsequent theorem proved by Randell in~\cite{Ran97} states that the Milnor 
fiber and fibration are also invariants for isotopic arrangements.

\begin{theorem}[\cite{Ran97}]\label{MFF}
 Let $\mathcal{A}_{t}$ be a smooth one-parameter family of central hyperplane arrangements. If $\mathcal{A}_{t}$ 
 is an isotopy, then for any $t_{1}$ and $t_{2}$ the Milnor fibrations $Q_{\mathcal{A}_{t_{1}}}$ and 
 $Q_{\mathcal{A}_{t_{2}}}$ are isomorphic fiber bundles.
\end{theorem}
   
\section{Reduced realization spaces}\label{Section2}

Throughout this section we suppose that, given a rank $d$ matroid $M$ with 
ground set $E =\{1,\ldots,m\}$, the set $\{1,\ldots,d\}$ is a basis of $M$. We 
can always assume this after relabelling the elements of the ground set.

Our goal is to introduce a subspace~$\mathcal{R}_{\mathbb{C}}^{R}(M)$ of the 
realization space~$\mathcal{R}_{\mathbb{C}}(M)$ that contains information about 
the realizability of~$M$ over $\mathbb{C}$ 
and the connectedness of~$\mathcal{R}_{\mathbb{C}}(M)$, 
but it is easier to describe than the full space~$\mathcal{R}_{\mathbb{C}}(M)$.

Suppose in fact that $A \in M_{d,m}(\mathbb{C})$ realizes~$M$ over~$\mathbb{C}$, 
namely $A \in \mathcal{R}_{\mathbb{C}}(M)$. Since we assume that $\{1, \ldots, 
d\}$ is a basis for~$M$, we can perform a change of coordinates 
in~$\mathbb{C}^d$ in such a way that the columns $A^1, \ldots, A^d$ of $A$ 
become the standard basis. 
The new matrix we obtain realizes~$M$ over $\mathbb{C}$ as well. At 
this point, we can multiply every row of~$A$ by a non-zero scalar without 
modifying the realizability property. Therefore, for a matrix $A \in 
M_{d,m}(\mathbb{C})$ realizing~$M$ over~$\mathbb{C}$ we can try to find an 
invertible matrix $G\in GL_{d}(\mathbb{C})$ of rank~$d$ and a complex 
non-singular diagonal matrix~$D$ of rank~$m$ so that $GAD$ has as many zeros and 
ones as possible, and still realizes~$M$ over~$\mathbb{C}$. Our new space will 
correspond to the set of these ``reduced'' matrices. 
To be more specific, we would like that the new matrix $GAD$ is of the form 
$\left(I_{d}\left|\tilde{A}\right.\right)$, where $I_{d}$ is the $d\times d$ 
identity matrix and $\tilde{A}\in M_{d,m-d}(\mathbb {C})$ is a matrix of 
$d$ rows and $m-d$ columns with complex coefficients that fulfills the 
following properties:
\begin{itemize}
 \item For each column of $\tilde{A}$, the first non-zero entry 
       (from the top to the bottom) equals $1$;
 \item For each row of $\tilde{A}$, the first non-zero entry 
       (from the left to the right) that is not the first non-zero entry 
       (from the top to the bottom) of a column equals $1$.
\end{itemize}

In order to define precisely and to be able to manipulate the object we are 
going to define, we need a somehow technical notion, the \emph{normal frame} of 
a matrix. This is a way to encode the ``support'' of a particular element of 
the equivalence class of a matrix~$Q$ under the left action 
by~$GL_d(\mathbb{C})$ and the right action by~$\mathbb{C}^{*}$. By ``support'' 
we mean that the entries in the normal frame have value~$1$ for such an element 
in the equivalence class. Let us consider a matrix $Q\in 
M_{n,r}(\mathbb{C})$ of $n$ rows and $r$ columns with complex coefficients and 
let us associate to~$Q$ a board~$S_{0}(Q)$ of $n$ rows and $r$ columns with 
black squares in correspondence to the zero entries of~$Q$ and white squares in 
correspondence to the non-zero entries of~$Q$. At this point, we perform the 
following sequence of operations on the board~$S_{0}(Q)$:
\begin{enumerate}[label=(O\arabic{*})]
 \item \label{O1} For each column of $S_{0}(Q)$ we color blue the first white square 
                  from the top to the bottom. We call this board $S_{1}(Q)$;
 \item \label{O2} For each row of $S_{1}(Q)$ we color red the first white square from 
                  the left to the right. We call this board $S_{2}(Q)$;
 \item \label{O3} We color green each blue or red square of $S_{2}(Q)$. We call this 
                  board $S(Q)$.
\end{enumerate}

\begin{definition}[Normal frame]\label{normalframe}
 The \emph{normal frame} of a matrix $Q\in M_{n,r}(\mathbb{C})$ is
 \begin{equation*}
 \mathcal{P}_{Q} = \bigl\{
                     (i,j)\in\{1,\ldots,n\}\times\{1,\ldots,r\} \, \mid \, 
\text{the 
                     $(i,j)\text{-th square}$ of $S(Q)$ is green}
                   \bigr\}
\end{equation*}
\end{definition}

We are now ready to define the reduced realization space of a matroid. 
 
\begin{definition}[Reduced realization space]\label{reducedrealizationspace}
 For a rank $d$ matroid $M$ with ground set $E=\{1,\ldots,m\}$ and 
 $\{1,\ldots,d\}$ as basis, the \emph{reduced realization space} of $M$ 
over~$\mathbb{C}$ is the set~$\mathcal{R}_{\mathbb{C}}^{R}(M)$ of matrices 
$A\in M_{d,m}(\mathbb{C})$ that satisfy the following 
 conditions:
 \begin{enumerate}[label=(C\arabic{*})]
  \item \label{C1} $A$ realizes $M$ over $\mathbb{C}$, that is, $A$ belongs to 
                   $\mathcal{R}_{\mathbb{C}}(M)$;
  \item \label{C2} $A$ is of the form $\left(I_{d}\left|\tilde{A}\right.\right)$, 
                   where $I_d$ is the $d \times d$ identity matrix;
  \item \label{C3} The entries of $\tilde{A}$ that are in the normal frame 
                   $\mathcal{P}_{\tilde{A}}$ equal $1$. 
 \end{enumerate}
 We endow~$\mathcal{R}_{\mathbb{C}}^{R}(M)$ with the subspace topology 
of~$M_{d,m}(\mathbb{C})$.
\end{definition}

\begin{remark}\label{KeyNote}
 For a matrix $A\in M_{d,m}(\mathbb{C})$, condition \ref{C1} is equivalent to 
 \begin{equation}
 \tag{$\ast$}\label{algebraicset}
  \left\lbrace
  \begin{array}{lll}
   \det\left(A^{j_{1}}|\cdots|A^{j_{d}}\right)\neq0 & \text{if} & 
            \left\lbrace j_{1},\ldots,j_{d}\right\rbrace\in\mathfrak{B} \\
   \det\left(A^{j_{1}}|\cdots|A^{j_{d}}\right)=0    & \text{if} & 
            \left\lbrace j_{1},\ldots,j_{d}\right\rbrace\notin\mathfrak{B} \\
  \end{array}
  \right.
 \end{equation}
 where $A^{j}$ denotes the $j\text{-th}$ column of $A$ and $\mathfrak{B}$ is the set of bases of $M$.
 For $1\leq i\leq d$ and $d+1\leq j\leq m$, 
if we consider the $d\text{-uples}$
\begin{equation*}
 (\{1,\ldots,d\}\setminus\{i\})\cup\{j\}
\end{equation*}
it follows from \eqref{algebraicset} that, given matrices $\tilde{A}_{1}$ and $\tilde{A}_{2}$ in $M_{d,m-d}(\mathbb{C})$ 
with $(I_{d}|\tilde{A}_{1})$ and $(I_{d}|\tilde{A}_{2})$ in $\mathcal{R}_{\mathbb{C}}^{R}(M)$, the board $S_{0}(\tilde{A}_{1})$ 
equals $S_{0}(\tilde{A}_{2})$. Hence, all matrices $\tilde{A}$ in  $M_{d,m-d}(\mathbb{C})$ with $(I_{d}|\tilde{A})$ in 
$\mathcal{R}_{\mathbb{C}}^{R}(M)$ have the same normal frame. 
This, together with condition \ref{C2} and \eqref{algebraicset}, 
implies that $\mathcal{R}_{\mathbb{C}}^{R}(M)$ can be written as a subset of $M_{d,m}(\mathbb{C})$ satisfying a system of 
equalities and inequalities of polynomial type.
\end{remark}

The subsequent proposition will clarify how the 
spaces $\mathcal{R}_{\mathbb{C}}(M)$ and $\mathcal{R}_{\mathbb{C}}^{R}(M)$ are 
related. In particular, it shows that the connectedness 
of $\mathcal{R}_{\mathbb{C}}^{R}(M)$ implies the one 
of $\mathcal{R}_{\mathbb{C}}(M)$. This fact is a direct consequence of the 
connectedness of the complex linear group and of the complex torus.

\begin{proposition}\label{keytool}
 For a rank $d$ matroid $M$ with ground set $E=\{1,\ldots,m\}$ and 
 $\{1,\ldots,d\}$ as basis, let $A\in\mathcal{R}_{\mathbb{C}}(M)$
 be a matrix that realizes~$M$ over~$\mathbb{C}$. 
 Then, there exist an invertible matrix $G\in GL_{d}(\mathbb{C})$ of rank~$d$ 
 and a complex non-singular diagonal 
 matrix $D$ of rank $m$ such that $GAD \in \mathcal{R}_{\mathbb{C}}^{R}(M)$. In 
 particular, the following properties hold:
 \begin{enumerate}[label=(P\arabic{*})]
  \item \label{P1} $\mathcal{R}_{\mathbb{C}}(M)\neq\emptyset$ if and only if 
                   $\mathcal{R}_{\mathbb{C}}^{R}(M)\neq\emptyset$;
  \item \label{P2} If $\mathcal{R}_{\mathbb{C}}^{R}(M)$ is connected, so is 
                   $\mathcal{R}_{\mathbb{C}}(M)$.
 \end{enumerate}
\end{proposition}

To show this result we need at first to prove some technical lemmas.

\begin{lemma}\label{boardframe}
 For a matrix $Q\in M_{n,r}(\mathbb{C})$ with at least a non-zero entry,
 consider the board~$S(Q)$ associated to~$Q$. Then, there exists a line 
 (row or column) of~$S(Q)$ that contains 
 exactly one green square and such that the board obtained from~$S(Q)$ by 
 deleting this line coincides with the one obtained from~$S_{0}(Q)$ by 
 deleting such line and then performing the steps
 \ref{O1}, \ref{O2} and \ref{O3}.
\end{lemma}

\begin{proof}
 Without loss of generality we can assume that each line (row or column) 
 of~$S_{0}(Q)$ contains at least a white square. 
 Otherwise, it suffices to delete that black line and study the problem for a 
 smaller board. Set
 \begin{equation*}
  \nu = \max \left\lbrace
               i\in\left\lbrace
                 1,\ldots,n
                 \right\rbrace
                 \, \mid \,
               \text{the $i\text{-th}$ row of $S_{1}(Q)$ contains a blue square}
             \right\rbrace
 \end{equation*}
 and notice that under the assumption that each line of~$S_{0}(Q)$ 
 contains at least a white square, this number~$\nu$ 
 is well defined. We distinguish two cases:
 \begin{itemize}
  \item If $1\leq\nu<n$, then the statement follows by considering the  
        $(\nu+1)\text{-th}$ row.
  \item If $\nu=n$, then it suffices to consider the first column for which 
        this maximum is attained. \qedhere
 \end{itemize}
\end{proof}

\begin{lemma}\label{matrixframe}
 Given a matrix $Q\in M_{n,r}(\mathbb{C})$ there exist complex non-singular 
 diagonal matrices $D_{1}$ of rank $n$ and $D_{2}$ of rank~$r$ such that the 
 entries of $D_{1}QD_{2}$ that belong to the normal frame $\mathcal{P}_{Q}$ 
 of~$Q$ equal~$1$.
\end{lemma}
\begin{proof}
Our proof exploits the same ideas of \cite[Proposition 2.7]{BL73}. We proceed 
by induction on the cardinality of the normal frame~$\mathcal{P}_{Q}$ of~$Q$. If 
$\left|\mathcal{P}_{Q}\right|=0$, then there is nothing to prove, since this condition 
is equivalent to the fact that all entries of~$Q$ are zero. Now, let us assume 
our statement true for all matrices with normal frame of cardinality strictly 
less than $k$ and let us consider a matrix~$Q$ with normal 
frame~$\mathcal{P}_{Q}$ of $k$ elements. From Lemma~\ref{boardframe} we know 
that there exists a line (row or column) of the board~$S(Q)$ that contains 
exactly one green square and such that the board obtained from~$S(Q)$ by 
deleting this line coincides with the one obtained from~$S_{0}(Q)$ and then 
performing the steps \ref{O1}, \ref{O2} and \ref{O3}. Notice that our proof will be 
essentially the same if we suppose that line is a column. Hence, let us assume 
that line is the $i\text{-th}$ row of~$S(Q)$. Let us denote by $(i,j)$ the 
position of the unique green square placed in it. In particular, the 
entry~$q_{ij}$ of~$Q$ is non-zero. Otherwise, by definition of the algorithmic 
steps \ref{O1}, \ref{O2} and \ref{O3} there will be a black square in the position 
$(i,j)$ of the board~$S(Q)$. Let us denote by $\tilde{Q} \in 
M_{n-1,r}(\mathbb{C})$ the matrix obtained from~$Q$ by deleting its 
$i\text{-th}$ row. With the second part of the statement of Lemma 
\ref{boardframe} we can deduce that the normal frame~$\mathcal{P}_{\tilde{Q}}$ 
of~$\tilde{Q}$ has $k-1$ elements. Thus, by inductive hypothesis there exist 
complex non-singular diagonal matrices~$\tilde{D}_{1}$ of rank~$n-1$ and 
$\tilde{D}_{2}$ of rank~$r$ such that the entries of 
$\tilde{D}_{1}\tilde{Q}\tilde{D}_{2}$ that belong to the normal 
frame~$\mathcal{P}_{\tilde{Q}}$ of~$\tilde{Q}$ equal~$1$. 
For $i\in\{1,2\}$, we denote by~$\tilde{D}_i(j)$ the $j$-th 
diagonal element of the matrix~$\tilde{D}_i$.
So finally, if we 
define 
 \begin{equation*}
 D_{1} = \operatorname{diag}\left(\tilde{D}_{1}(1),\cdots,\tilde{D}_{1}(i-1),
\left(\tilde{D}_{2}(j)q_{ij}\right)^{ -1},                           
\tilde{D}_{1}(i),\cdots,\tilde{D}_{1}(n-1) \right)
 \end{equation*}
and set $D_{2}=\tilde{D}_{2}$, one can easily check that all entries 
of $D_{1}QD_{2}$ that belong to the normal  
frame~$\mathcal{P}_{Q}$ of~$Q$ are equal to~$1$.
\end{proof}

\begin{proof}[Proof of Proposition \ref{keytool}]
 Let $A\in\mathcal{R}_{\mathbb{C}}(M)$ be a matrix that realizes $M$ 
 over $\mathbb{C}.$ Since $\{1,\ldots,d\}$ is a basis of~$M$, there exists an 
 invertible matrix $B\in GL_{d}(\mathbb{C})$ of rank~$d$ such that 
 $BA=\left(I_{d}\left|Q\right.\right)$, where $Q\in M_{d,m-d}(\mathbb{C})$. By 
 Lemma \ref{matrixframe} there exist complex non-singular diagonal 
 matrices~$D_{1}$ of rank~$d$ and~$D_{2}$ of rank~$m-d$ such that the entries 
 of $D_{1}QD_{2}$ belonging to the normal frame~$\mathcal{P}_{Q}$ of~$Q$ are
 equal to~$1$. Now, for $i\in\{1,2\}$ denote by~$D_i(j)$ the $j$-th 
 diagonal element of the matrix~$D_i$, and set 
 \begin{equation*}
  D = \operatorname{diag} \left(D_{1}(1)^{-1}, \ldots,D_{1}(d)^{-1}, D_{2}(1), 
\ldots,D_ {2}(m-d) \right)
 \end{equation*}
 and $G=D_{1}B$. With elementary linear algebra arguments it is not hard to see 
 that the matrix $GAD$ realizes the matroid~$M$ over~$\mathbb{C}$ as well. 
 Hence, condition~\ref{C1} in Definition \ref{reducedrealizationspace} is 
 satisfied. By construction the matrix $GAD$ is of the form 
 $\left(I_{d}\left|D_{1}QD_{2}\right.\right)$, so that conditions \ref{C2} 
 and \ref{C3} in Definition \ref{reducedrealizationspace} are fulfilled, too.

\smallskip
 Hence, it remains to check that properties \ref{P1} and \ref{P2} hold.
 \begin{itemize}
  \item Property \ref{P1} follows directly from the first part of our statement 
and the set inclusion 
        $\mathcal{R}_{\mathbb{C}}^{R}(M)\subseteq\mathcal{R}_{\mathbb{C}}(M)$.
  \item To prove that \ref{P2} is satisfied, let us assume 
        $\mathcal{R}_{\mathbb{C}}^{R}(M)$ connected. We show that under this 
        assumption $\mathcal{R}_{\mathbb{C}}(M)$ is actually a path connected 
        space. Since 
        $\mathcal{R}_{\mathbb{C}}^{R}(M)$ can be expressed as a subset 
        of~$M_{d,m}(\mathbb{C})$ satisfying a system of 
        polynomial equalities and inequalities (see Remark \ref{KeyNote}), the connectedness hypothesis 
        of~$\mathcal{R}_{\mathbb{C}}^{R}(M)$ implies that 
        $\mathcal{R}_{\mathbb{C}}^{R}(M)$ is path connected. 
        Let $A$ and $B$ be matrices of~$\mathcal{R}_{\mathbb{C}}(M)$. Using the 
        first part of 
        our statement, let $G_{1}$ and $G_{2}\in GL_{d}(\mathbb{C})$ be invertible 
        matrices of rank~$d$ 
        and let $D_{1}$, $D_{2}$ be complex
        non-singular diagonal matrices of rank~$m$ 
        such that $G_{1}AD_{1}$ and $G_{2}BD_{2}$ belong 
        to~$\mathcal{R}_{\mathbb{C}}^{R}(M)$. Since
        $\mathcal{R}_{\mathbb{C}}^{R}(M)$ is path connected, 
        we can find a continuous path 
        $\gamma \colon \left[0,1\right] \longrightarrow 
        \mathcal{R}_{\mathbb{C}}^{R}(M)$ such that $\gamma(0)$ equals $G_{1}AD_{1}$ and 
        $\gamma(1)$ equals $G_{2}BD_{2}$. 
        Moreover, from the inclusion $\mathcal{R}_{\mathbb{C}}^{R}(M) 
        \subseteq \mathcal{R}_{\mathbb{C}}(M)$ 
        and the fact that both these spaces are endowed with the 
        subspace topology of~$M_{d,m}(\mathbb{C})$, we see that 
        $\gamma$ is indeed a continuous path with values 
        in~$\mathcal{R}_{\mathbb{C}}(M)$.
        The complex linear group~$GL_{d}(\mathbb{C})$ is path connected, since 
        it is the 
        complement of the complex hypersurface
        $\left\lbrace\det(X)=0\right\rbrace$ in $M_{d,d}(\mathbb{C})$.
        Also the space~$D_{m}(\mathbb{C})$ of complex non-singular diagonal 
        matrices of rank~$m$ is path connected, 
        since it can be diffeomorphically identified with the complex 
        torus $\left(\mathbb{C}^{\ast}\right)^{m}$.
        Thus, there exist continuous paths 
        \begin{equation*}
         \sigma_{1},\sigma_{2} \colon \left[0,1\right]\longrightarrow 
         GL_{d}(\mathbb{C})
         \quad \text{and} \quad
         \tau_{1},\tau_{2} \colon \left[0,1\right]\longrightarrow 
         D_{m}(\mathbb{C})
        \end{equation*}
        with 
        \begin{equation*}
         \begin{array}{llll}
          \sigma_{1}(0)=I_{d}, & \sigma_{1}(1)=G_{1}, & \sigma_{2}(0)=I_{d}, & 
          \sigma_{2}(1)=G_{2}, \\
          \tau_{1}(0)=I_{m},   & \tau_{1}(1)=D_{1},   & \tau_{2}(0)=I_{m},   & 
          \tau_{2}(1)=D_{2}.   \\
         \end{array}
        \end{equation*}
        Now, consider $\Gamma_{A}(t) = \sigma_{1}(t)A\tau_{1}(t)$ and 
        $\Gamma_{B}(t)=\sigma_{2}(t)B\tau_{2}(t)$. Again, using elementary 
        linear algebra arguments, we can easily see that for 
        $t\in\left[0,1\right]$ the matrices $\Gamma_{A}(t)$ and 
        $\Gamma_{B}(t)$ belong to $\mathcal{R}_{\mathbb{C}}(M)$. 
        So finally, if we consider the joined path
        \begin{equation*}
         \sigma(t)=\left\lbrace
                   \begin{array}{lll}
                    \Gamma_{A}(3t)   & \text{if} & t\in\left[0,1/3\right]    \\
                    \gamma(3t-1)     & \text{if} & t\in\left[1/3,2/3\right]  \\
                    \Gamma_{B}(3-3t) & \text{if} & t\in\left[2/3,1\right]    \\
                   \end{array}
                   \right.
        \end{equation*}
        we obtain a continuous path $\sigma 
        \colon \left[0,1\right]\longrightarrow\mathcal{R}_{\mathbb{C}}(M)$ with 
        $\sigma(0)=A$ and $\sigma(1)=B$. \qedhere
 \end{itemize}
\end{proof}

\section{Applications}\label{Section3}

The aim of this section is to prove that complex central hyperplane arrangements 
with up to~$7$ hyperplanes and same underlying matroid are isotopic, improving 
the results of~\cite{NY12} and~\cite{Ye13} to any rank. The central idea of our 
proof is to exploit the connectedness of the reduced realization space of the 
underlying matroid of these arrangements to apply Proposition~\ref{keytool}.

\begin{theorem}\label{main}
 Let $\mathcal{A}=\{H_{1},\ldots,H_{m}\}$ and 
 $\mathcal{B}=\{K_{1},\ldots,K_{m}\}$ be central essential hyperplane 
 arrangements in $\mathbb{C}^{d}$ with same underlying matroid. If $m\leq 7$, 
 then $\mathcal{A}$ and $\mathcal{B}$ are isotopic arrangements.
\end{theorem}

This result implies that the diffeomorphism type of the complement manifold and 
the Milnor fiber and fibration of these arrangements are uniquely determined by 
their underlying matroid. 

\begin{corollary}
 Let $\mathcal{A}=\{H_{1},\ldots,H_{m}\}$ and  
 $\mathcal{B}=\{K_{1},\ldots,K_{m}\}$ be central essential hyperplane 
 arrangements in $\mathbb{C}^{d}$ with same underlying matroid. If $m\leq 7$, 
 then the following properties are fulfilled:
 \begin{enumerate}
  \item The complement manifolds $M(\mathcal{A})$ and $M(\mathcal{B})$ are 
diffeomorphic;
  \item The Milnor fibrations $Q_{\mathcal{A}}$ and $Q_{\mathcal{B}}$ are 
isomorphic fiber bundles. 
 \end{enumerate}
\end{corollary}

\begin{proof}
 $(1)$ follows from Theorem~\ref{RIT} and $(2)$ is a consequence of 
Theorem~\ref{MFF}. 
\end{proof}

To prove Theorem~\ref{main} some preliminary results are required.

\begin{lemma}\label{path}
 For a rank~$d$ matroid~$M$ with ground set $E=\{1,\ldots,m\}$ let $A$ and $B$ 
 be two matrices that  realize~$M$ 
 over $\mathbb{C}$ and belong to the same connected component 
 of $\mathcal{R}_{\mathbb{C}}(M)$. Then, there exists 
 $\epsilon>0$ and a smooth path 
 $\sigma \colon (-\epsilon,1+\epsilon)\longrightarrow M_{d,m}(\mathbb{C})$ such 
 that $\sigma(0)=A$, 
 $\sigma(1)=B$ and $\sigma(t)\in\mathcal{R}_{\mathbb{C}}(M)$ for 
$t\in(-\epsilon,1+\epsilon)$.
\end{lemma}
\begin{proof}
 Let $\mathfrak{B}$ be the set of bases of $M$ and write
 \begin{equation*}
  \mathcal{R}_{\mathbb{C}}(M)=\left\lbrace
                              A\in M_{d,m}(\mathbb{C}) \,
                                \left|
                                 \begin{array}{lll}
                                  
\det\left(A^{j_{1}}|\cdots|A^{j_{d}}\right)\neq0 & \text{if} 
                                                  & \left\lbrace 
j_{1},\ldots,j_{d}\right\rbrace\in\mathfrak{B}    \\
                                  \det\left(A^{j_{1}}|\cdots|A^{j_{d}}\right)=0  
  & \text{if} 
                                                  & \left\lbrace 
j_{1},\ldots,j_{d}\right\rbrace\notin\mathfrak{B} \\                             
                             
                                 \end{array}
                                \right.
                              \right\rbrace
 \end{equation*}
 where $A^{j}$ denotes the $j\text{-th}$ column of~$A$.
 Hence, the space $\mathcal{R}_{\mathbb{C}}(M)$ can be expressed 
 as a subset of $M_{d,m}(\mathbb{C})$ satisfying a system of 
 equalities and inequalities of polynomial type. 
 As a consequence of this, each connected component~$\mathcal{C}$ 
of~$\mathcal{R}_{\mathbb{C}}(M)$ is actually a piecewise linear path 
connected space. Thus, the following property is fulfilled:
 \begin{center}
  If $A, B\in\mathcal{C}$, then there exists $\epsilon>0$ and a piecewise 
  linear path $\gamma \colon (-\epsilon,1+\epsilon) \longrightarrow 
  \mathcal{C}$ with $\gamma(0)=A$ and $\gamma(1)=B$. 
 \end{center}
Since the equalities and inequalities that define~$\mathcal{R}_{\mathbb{C}}(M)$ 
are of polynomial type, it is not hard to see that it is possible to
reparametrize the path~$\gamma$ by stopping of infinite order at each point 
where it it is not smooth
(using pieces $t\mapsto e^{-\frac{1}{t^{2}}}$)
in order 
to find a smooth path $\sigma\colon(-\epsilon,1+\epsilon)\longrightarrow M_{d,m}(\mathbb{C})$
such that $\sigma(0)=A$, $\sigma(1)=B$ and $\sigma(t)\in\mathcal{C}$ for 
$t \in (-\epsilon,1+\epsilon)$.
\end{proof}

\begin{lemma}\label{small}
 For a rank $d$ matroid~$M$ with ground set $E=\{1,\ldots,m\}$ and 
 $\{1,\ldots,d\}$ as basis, if $1\leq d\leq m\leq 7$ and 
 $M$ is realizable over~$\mathbb{C}$, then the
 space~$\mathcal{R}_{\mathbb{C}}^{R}(M)$ is non-empty and connected.
\end{lemma}
\begin{proof}
 Since by hypothesis $M$ is realizable over $\mathbb{C}$, we have that 
$\mathcal{R}_{\mathbb{C}}(M)$ is non-empty, and so by Proposition~\ref{keytool} 
we get $\mathcal{R}_{\mathbb{C}}^{R}(M) \neq \emptyset$. 
The space $\mathcal{R}_{\mathbb{C}}^{R}(M)$ can be expressed as a subset of $M_{d,m}(\mathbb{C})$ 
satisfying a system of polynomial equalities and inequalities (see Remark \ref{KeyNote}). 
From \cite[Chapter~7, Theorem~7.1]{Sha13} 
to prove connectedness it is enough to show that~$\mathcal{R}_{\mathbb{C}}^{R}(M)$ is 
irreducible in the Zariski topology. We checked this for all matroids $M$ 
satisfying the hypothesis by a direct computation with the aid of the computer 
algebra system Sage~\cite{sage} (for further details, see 
Appendix~\ref{AppendixA}). 
\end{proof}

\begin{proof}[Proof of Theorem \ref{main}]
 It is clear that the only interesting case is when $d\geq1$ and $m\geq1$.
 To prove our statement we have to distinguish between the two cases 
 $1\leq d\leq m$ and $1\leq m<d$.
 
 \smallskip
 \noindent \emph{Case $1\leq d\leq m$.} Let $M$ be the underlying matroid 
 of the arrangements~$\mathcal{A}$ and~$\mathcal{B}.$ Up to relabelling the 
 hyperplanes~$\{ H_i \}$ of~$\mathcal{A}$ and~$\{ K_i \}$ of~$\mathcal{B}$, 
 let us suppose that $\{1,\ldots,d\}$ is a basis of $M$. 
 Notice that we can always do this, 
 since $\mathcal{A}$ and $\mathcal{B}$ are essential arrangements.
 Pick linear forms~$\alpha_{i}$ and~$\beta_{i}$ such that 
 $H_{i}=\ker\alpha_{i}$ and $K_{i}=\ker\beta_{i}$. Let us denote 
 by $\alpha_{i}^{j}$ and $\beta_{i}^{j}$ the $j\text{-th}$ component 
 of $\alpha_{i}$ and $\beta_{i},$ respectively. Set $A = (\alpha_{i}^{j})^{t}$ 
 and $B = (\beta_{i}^{j})^{t}.$ Now, consider the 
 space $\mathcal{R}_{\mathbb{C}}(M).$ The matrices~$A$ and~$B$ belong to 
 $\mathcal{R}_{\mathbb{C}}(M).$ Hence, to prove that $\mathcal{A}$ and 
 $\mathcal{B}$ are isotopic arrangements (compare Definition~\ref{isotopic}) it 
 is enough to show that there exists $\epsilon>0$ and a smooth path $\sigma 
 \colon (-\epsilon,1+\epsilon) \longrightarrow M_{d,m}(\mathbb{C})$ with 
 $\sigma(0)=A,$ $\sigma(1)=B$ and $\sigma(t)\in\mathcal{R}_{\mathbb{C}}(M)$ for 
 $t$ in $(-\epsilon,1+\epsilon)$. Thus, with Lemma \ref{path} it suffices to check 
 that $\mathcal{R}_{\mathbb{C}}(M)$ is connected. To see this, thanks to 
 Proposition~\ref{keytool}, we can just verify the connectedness 
 of~$\mathcal{R}_{\mathbb{C}}^{R}(M).$ So that, the statement follows from  
 Lemma~\ref{small}.
 
 \smallskip
 \noindent \emph{Case $1\leq m<d$.} This follows from elementary complex 
 linear algebra arguments. As previously done, for the hyperplanes $\{H_{i}\}$ of $\mathcal{A}$ 
 and $\{K_{i}\}$ of $\mathcal{B}$ choose linear forms 
 $\alpha_{i}$ and $\beta_{i}$ with 
 $H_{i}=\ker\alpha_{i}$ and $K_{i}=\ker\beta_{i}$. Let us denote 
 by $\alpha_{i}^{j}$ and $\beta_{i}^{j}$ the $j\text{-th}$ component 
 of $\alpha_{i}$ and $\beta_{i},$ respectively. Set $A = (\alpha_{i}^{j})^{t}$ 
 and $B = (\beta_{i}^{j})^{t}$. Now, consider the space
 \begin{equation*}
  \mathcal{S}_{d,m}(\mathbb{C})=\left\lbrace
                                 Q\in M_{d,m}(\mathbb{C})\mid\operatorname{rk}Q=m
                                \right\rbrace
 \end{equation*}
 and notice that $A$ and $B$ belong to $\mathcal{S}_{d,m}(\mathbb{C})$ since the 
 arrangements~$\mathcal{A}$ and~$\mathcal{B}$ are essential. Again, to show 
 that $\mathcal{A}$ and $\mathcal{B}$ are isotopic arrangements it 
 suffices to prove that there is $\epsilon>0$ and a smooth path 
 $\sigma\colon(-\epsilon,1+\epsilon)\longrightarrow M_{d,m}(\mathbb{C})$ with 
 $\sigma(0)=A$, $\sigma(1)=B$ and $\sigma(t)\in\mathcal{S}_{d,m}(\mathbb{C})$ 
 for~$t$ in~$(-\epsilon,1+\epsilon)$. With the same arguments of 
 Lemma~\ref{path}, it is enough to verify the connectedness 
 of~$\mathcal{S}_{d,m}(\mathbb{C})$.
 To prove this, let us write
 \begin{equation*}
  \mathcal{S}_{d,m}(\mathbb{C})=\bigcup_{1\leq i_{1}<\cdots<i_{m}\leq d}
  \mathcal{S}_{d,m}^{i_{1}\cdots i_{m}}(\mathbb{C})
 \end{equation*}
 where 
 \begin{equation*}
  \mathcal{S}_{d,m}^{i_{1}\cdots i_{m}}(\mathbb{C})=\left\lbrace
                                                    Q\in M_{d,m}(\mathbb{C})
                                                    \mid
                                                    \det(Q_{i_{1}}^{t}|\cdots|Q_{i_{m}}^{t})\neq0
                                                   \right\rbrace
 \end{equation*}
 and $Q_{i}$ stands for the $i\text{-th}$ row of~$Q$.
 Each space $\mathcal{S}_{d,m}^{i_{1}\cdots i_{m}}(\mathbb{C})$ is path connected, since it is the 
 complement of the complex hypersurface 
 $\{\det(Q_{i_{1}}^{t}|\cdots|Q_{i_{m}}^{t})=0\}$ in~$M_{d,m}(\mathbb{C})$.
 Hence, to conclude our proof it is sufficient to show that 
 \begin{equation*}
  \bigcap_{1\leq i_{1}<\cdots<i_{m}\leq d}\mathcal{S}_{d,m}^{i_{1}\cdots i_{m}}(\mathbb{C})\neq\emptyset
 \end{equation*}
 and this is equivalent to say that 
 \begin{equation*}
  \prod_{1\leq i_{1}<\cdots<i_{m}\leq d}\det(Q_{i_{1}}^{t}|\cdots|Q_{i_{m}}^{t})
 \end{equation*}
 is not the zero polynomial. 
 None of the factors $\det(Q_{i_{1}}^{t}|\cdots|Q_{i_{m}}^{t})$ is the zero polynomial. Thus, the statement 
 follows from the fact that the ring of polynomials in~$dm$ variables with 
 complex coefficients is an integral domain.
\end{proof}

\begin{remark}
 Notice that in the case $1 \leq m<d$ we never used the fact that $m \leq 7$, 
and so in this situation the result of Theorem~\ref{main} holds without any 
numerical restriction.
\end{remark}

\appendix
 
\section{Checking connectedness of reduced realization spaces}\label{AppendixA}

We are going to show by a direct test that Lemma~\ref{small} holds. 
For a rank~$d$ matroid~$M$ with ground set $E=\{1,\ldots,m\}$ and 
$\{1,\ldots,d\}$ as basis, let us consider a matrix $G_{0,M}\in M_{d,m}(\mathbb{C})$ 
with all entries equal to $-1$ and let us perform the following 
sequence of operations:
\begin{enumerate}[label=(S\arabic{*})]
 \item \label{S1} We insert a $d\times d$ identity matrix in correspondence of the 
                  first $d$ columns of $G_{0,M}$. We call this matrix $G_{1,M}$;
 \item \label{S2} For $1\leq i\leq d$ and $d+1\leq j\leq m$ we set the 
                  $(i,j)\text{-th}$ entry of $G_{1,M}$ 
                  equal $0$ if $(\{1,\ldots,d\}\setminus\{i\})\cup\{j\}$ is not a 
                  basis of $M$. 
                  We call this matrix $G_{2,M}$;
 \item \label{S3} Let $\tilde{G}_{2,M}$ be the $d\times(m-d)$ matrix such that 
       $G_{2,M}=(I_{d}|\tilde{G}_{2,M})$. We set the entries of 
       $\tilde{G}_{2,M}$ that are 
       in the normal frame $\mathcal{P}_{\tilde{G}_{2,M}}$ equal $1$. 
       We call this matrix $\tilde{G}_{3,M}$ and we set 
       $G_{3,M}=(I_{d}|\tilde{G}_{3,M})$;
 \item \label{S4} We call $s_{M}$ the number of $-1$ entries of $G_{3,M}$;
 \item \label{S5} We replace the $-1$ entries of $G_{3,M}$ with symbolic 
       variables $t_{1},\ldots,t_{s_{M}}$ and we call 
       this matrix $G_{M}$.
\end{enumerate}

\begin{algorithm}
\caption{$\mathtt{TestIrreducibility}$\label{alg1}}
\begin{algorithmic}[1]
  \Require $\mathtt{case} = (d,m)$ a pair from Equation~\eqref{cases}.
  \Ensure $\mathtt{True}$ if the reduced realization spaces of all realizable 
matroids of type $\mathtt{case}$ are irreducible, $\mathtt{False}$ otherwise.
  \Statex
  \State {\bfseries Compute} the list $\mathtt{subsets}$ of all subsets of~$d$ 
elements of $\{1, \dotsc, m\}$ and order it w.r.t.\ the reverse 
lexicographic 
term order. 
  \For {$\mathtt{matroid}$ in $\mathtt{all\_matroids[case]}$}
    \State {\bfseries Compute} the first basis for 
$\mathtt{matroid}$ in the 
list $\mathtt{subsets}$ and call it $\mathtt{basis}$.
    \State {\bfseries Set} $G = \mathtt{FillMatrix}(\mathtt{case}, 
\mathtt{basis})$.
    \State \Comment {Computing the (in)equalities for $X_{\mathtt{matroid}}$}.
    \State {\bfseries Substitute} the $-1$ entries of $G$ with symbolic 
variables.
    \State {\bfseries Set} $\mathtt{equalities} = \mathtt{emptylist}$ and 
$\mathtt{inequalities} = \mathtt{emptylist}$.
    \For {$\mathtt{subset}$ in $\mathtt{subsets}$}
      \State {\bfseries Set} $\mathtt{det}$ to be the $d \times d$ minor 
corresponding to the submatrix of $G$ whose columns are prescribed by 
$\mathtt{subset}$. 
      \If {$\mathtt{subset}$ is a basis for $\mathtt{matroid}$} {\bfseries 
Add} $\mathtt{det}$ to $\mathtt{inequalities}$.
      \Else \ {\bfseries Add} $\mathtt{det}$ to $\mathtt{equalities}$.
      \EndIf
    \EndFor
    \State \Comment {Checking irreducibility of the zero set determined by only 
the equalities.}
    \State {\bfseries Set} $\mathtt{ideal}$ to be the ideal generated by 
$\mathtt{equalities}$.
    \If {the zero set of $\mathtt{ideal}$ is not geometrically irreducible}
      \State \Return $\mathtt{False}$.
    \EndIf
  \EndFor
  \State \Return $\mathtt{True}$.
\end{algorithmic}
\end{algorithm}

\begin{definition}\label{reducedvariety}
 For a rank $d$ matroid $M$ with ground set $E=\{1,\ldots,m\}$ and $\{1,\ldots,d\}$ as basis the \emph{reduced variety}
 of $M$ over $\mathbb{C}$ is the quasi-projective variety $X_{M}$ defined by
 \begin{equation*}
 X_{M}=\left\lbrace
       \left(
        z_{1},\ldots,z_{s_{M}}
       \right)
       \in\mathbb{A}_{\mathbb{C}}^{s_{M}}
       \left|
        \begin{array}{lll}
         \det\left(G_{M}^{j_{1}}|\cdots|G_{M}^{j_{d}}\right)\neq0 & \text{if} & \{j_{1},\ldots,j_{d}\}\in\mathfrak{B}    \\
         \det\left(G_{M}^{j_{1}}|\cdots|G_{M}^{j_{d}}\right)=0    & \text{if} & \{j_{1},\ldots,j_{d}\}\notin\mathfrak{B} \\
        \end{array}
       \right.
       \right\rbrace
\end{equation*}
where $G_{M}^{j}$ is the $j\text{-th}$ column of $G_{M}$ and $\mathfrak{B}$ denotes the set of bases of $M$.
\end{definition}

\begin{remark}\label{rational}
 The defining equalities and inequalities of $X_{M}$ have integer coefficients.
\end{remark}

\begin{algorithm}
\caption{$\mathtt{FillMatrix}$}
\label{alg2}
\begin{algorithmic}[1]
  \Require $\mathtt{case} = (d,m)$, a pair from Equation~\eqref{cases}; 
$\mathtt{basis}$, a subset of $\{1, \dotsc, m\}$ of cardinality~$d$.
  \Ensure a matrix $G$, filled with entries belonging to $\{-1,0,1\}$ and 
ensuring Conditions \ref{C2} and \ref{C3} from Definition~\ref{reducedrealizationspace}.
  \Statex
    \State {\bfseries Create} a $d \times m$ matrix~$G$, and fill it with $-1$ 
entries.
    \State {\bfseries Set} $\mathtt{non\_basis}$ to be equal to the set $\{1, 
\dotsc, m\} \setminus \mathtt{basis}$.
    \State \Comment{Imposing Condition \ref{C2}.}
    \State {\bfseries Insert} in $G$ a $d \times d$ identity matrix in 
correspondence to the columns of $\mathtt{basis}$.
    \State \Comment {Inserting as many zeroes as possible in~$G$.}
    \For {$j$ in $\mathtt{non\_basis}$}
      \For {$i \in \{1, \dotsc, d\}$}
        \If {$\bigl( \{1, \dotsc, d\} \setminus \{i\} \bigr) \cup \{j\}$ is 
not a basis of $\mathtt{matroid}$} 
          \State {\bfseries Set} $G(i,j) = 0$.
        \EndIf
      \EndFor
    \EndFor
    \State \Comment {Computing the normal frame and imposing Condition \ref{C3}.}
    \State \Comment {Inserting $1$s column by column.}
    \For {$j$ in $\mathtt{non\_basis}$}
      \State {\bfseries Set} $r = 1$.
      \While {$G(r,j) = 0$}
        \State {\bfseries Increase} $r$ by~$1$.
        \If {$r = d+1$} {\bfseries Break} the loop. 
        \EndIf
      \EndWhile
      \If {$r \leq d$} {\bfseries Set} $G(r,j) = 1$.
      \EndIf
    \EndFor
    \State \Comment {Inserting $1$s row by row.}
    \For {$i \in \{1, \dotsc, d\}$}
      \State {\bfseries Set} $c = 1$.
      \While {$G(i,c) = 0$ or $G(i,c) = 1$}
        \State {\bfseries Increase} $c$ by~$1$.
        \If {$c = m+1$} {\bfseries Break} the loop. 
        \EndIf
      \EndWhile
      \If {$c \leq m$} {\bfseries Set} $G(i,c) = 1$.
      \EndIf
    \EndFor
  \State \Return $G$.
\end{algorithmic}
\end{algorithm}

If we compare Definition~\ref{reducedvariety} and 
Definition~\ref{reducedrealizationspace}, it is not hard to see that the 
quasi-projective variety~$X_{M}$ is isomorphic to the 
space~$\mathcal{R}_{\mathbb{C}}^{R}(M)$ 
endowed with the Zariski topology (see Remark~\ref{KeyNote} for more details).

Taking this into account, from now on we will be concerned with the 
determination of the irreducibility of~$X_M$. Notice that if $d = 1$, $d 
= m$, or $d = m-1$ the reduced variety $X_M$ is either empty (in which 
case $\mathcal{R}_{\mathbb{C}}^{R}(M) = \emptyset$, and so by 
Proposition~\ref{keytool} the matroid~$M$ is non-realizable over~$\mathbb{C}$), 
or equals a point (thus in particular $\mathcal{R}_{\mathbb{C}}^{R}(M)$ is irreducible).

Hence we are left with the cases when $(d,m)$ belongs to
\begin{equation}
\tag{$\ast\ast$}
\label{cases}
  \bigl\{ (2,4), (2,5), (2,6), (2,7), (3,5), (3,6), (3,7), 
(4,6), (4,7), (5,7) \bigr\}
\end{equation}
All matroids in these cases are classified (see~\cite{Matsumoto2011}) and the 
tables describing them are available at
\begin{center}
  \url{http://www-imai.is.s.u-tokyo.ac.jp/~ymatsu/matroid/index.html}
\end{center}

For all matroids~$M$ in the cases covered by Equation~\eqref{cases} we computed 
the equalities and inequalities defining~$X_M$. Notice that, if $\hat{X}_{M}$ is the subset 
of~$\mathbb{A}_{\mathbb{C}}^{s_{M}}$ defined by 
\begin{equation*}
 \hat{X}_{M}=\left\lbrace
       \left(
        z_{1},\ldots,z_{s_{M}}
       \right)
       \in\mathbb{A}_{\mathbb{C}}^{s_{M}}
       \left|
        \begin{array}{lll}
         \det\left(G_{M}^{j_{1}}|\cdots|G_{M}^{j_{d}}\right)=0    & \text{if} & \{j_{1},\ldots,j_{d}\}\notin\mathfrak{B} \\
        \end{array}
       \right.
       \right\rbrace
\end{equation*}
where $G_{M}^{j}$ is the $j\text{-th}$ column of $G_{M}$ and $\mathfrak{B}$ is the set of bases of $M$,
then by elementary topology arguments the 
irreducibility of~$\hat{X}_{M}$ implies the one of~$X_M$, if the latter is non-empty.
We checked by that $\hat{X}_{M}$ is always irreducible, hence we conclude 
that $X_M$ is always either empty, or irreducible. There are algorithms that 
decide whether an algebraic set defined by rational equalities 
(as~$\hat{X}_{M}$, recall Remark~\ref{rational}) 
is irreducible or not (see for example~\cite{Cheze2005}); in our case, via a 
direct inspection helped by computations with Sage, we noticed that all 
sets~$\hat{X}_{M}$ fall into one of these families:
\begin{itemize}
 \item Linear varieties;
 \item Rational hypersurfaces;
 \item Quadrics of rank strictly bigger than~$2$;
\end{itemize}
or are cones over such varieties, and so are irreducible by easy algebraic 
geometry arguments.

\smallskip
The Sage code we used to perform the test is available at the following links:
\begin{enumerate}
 \item http://perso.unifr.ch/elia.saini/hyperplanes.sage
 \item http://matteogallet.altervista.org/main/papers/hyperplanes2015/hyperplanes.sage
\end{enumerate}

The algorithm $\mathtt{TestIrreducibility}$ provided in Algorithm~\ref{alg1} 
describes the pseudocode of the main procedure we implemented and the algorithm 
$\mathtt{FillMatrix}$ presented in Algorithm~\ref{alg2} sketches the pseudocode 
of the ancillary algorithm we used to build the matrix~$G_{M}$ 
(compare the definition of the operations \ref{S1}, \ref{S2}, \ref{S3}, \ref{S4} 
and \ref{S5}).

\begin{remark}
 Notice that the assumption $m \leq 7$ does not play any role in any of the 
algorithms we presented in this Appendix. The only reason to limit ourselves to 
the case $m \leq 7$ is due to the fact that when $m$ is greater than~$7$ first 
of all the total number of matroids becomes significantly bigger, and moreover 
both the number and the degree of the equalities and inequalities 
defining~$X_M$ increases. This implies that the computations whose aim is to 
check whether $X_M$ is irreducible become more and more expensive in terms of 
memory and time, and moreover the cases when $\hat{X}_M$ does not fall into one 
of the simple families of varieties reported above become much more frequent.
Hence one would need to improve the existing algorithm and to find new families 
of algebraic varieties that can ensure irreducibility in order to attack the 
cases when $m > 7$, taking also into account the already known cases of 
matroids for which the variety~$X_M$ is reducible.
\end{remark}

\bibliographystyle{amsalpha}
\bibliography{GS_arXiv_final}

\end{document}